\documentclass{amsart}
\usepackage{amsmath}
\usepackage{amscd}
\usepackage{amssymb}

 \newtheorem{thm}{Theorem}[section]
 \newtheorem{cor}[thm]{Corollary}
 
 \newtheorem{prop}[thm]{Proposition}
 \theoremstyle{definition}
 \newtheorem{defn}[thm]{Definition}
 \newtheorem{ex}[thm]{Example}
 
 \newtheorem*{finrem}{Final remark}

 \numberwithin{equation}{section}

\newcommand{\bN}{\mathbb N}
\newcommand{\bZ}{\mathbb Z}

\newcommand{\bR}{\mathbb R}
\newcommand{\bC}{\mathbb C}
\newcommand{\bT}{\mathbb T}
\newcommand{\Id}{\operatorname{Id}}
\newcommand{\End}{\operatorname{End}}
\newcommand{\Per}{\operatorname{Per}}
\newcommand{\Det}{\operatorname{Det}}
\renewcommand{\Re}{\operatorname{Re}}

\newcommand{\card}{\operatorname{card}}
\newcommand{\Ns}{\operatorname{Ns}}
\newcommand{\Mon}{\operatorname{Mon}}
\newcommand{\fb}{\boldsymbol{f}}
\newcommand{\gb}{\boldsymbol{g}}

\newcommand{\fbs}{{}^{*\!}\boldsymbol{f}}
\newcommand{\gbs}{{}^{*\!}\boldsymbol{g}}
\newcommand{\fbo}{{}^{\circ\!}\boldsymbol{f}}
\newcommand{\ub}{\boldsymbol{u}}
\newcommand{\vb}{\boldsymbol{v}}
\newcommand{\xb}{\boldsymbol{x}}
\newcommand{\yb}{\boldsymbol{y}}
\newcommand{\zb}{\boldsymbol{z}}
\newcommand{\Ab}{\boldsymbol{A}}

\newcommand{\Vb}{\mathbf{V}}
\newcommand{\Vbs}{{}^*\mathbf{V}}
\newcommand{\B}{\mathcal{B}}
\newcommand{\C}{\mathcal{C}}
\newcommand{\F}{\mathcal{F}}
\newcommand{\K}{\mathcal{K}}
\newcommand{\T}{\mathcal{T}}
\newcommand{\U}{\mathcal{U}}

\newcommand{\AAf}{\mathfrak{A}}
\newcommand{\BBf}{\mathfrak{B}}
\newcommand{\XX}{\mathfrak{X}}
\newcommand{\YY}{\mathfrak{Y}}
\newcommand{\AAst}{{}^*\mathfrak{A}}

\newcommand{\Ast}{{}^{*\!}A}

\newcommand{\Cst}{{}^{*}C}
\newcommand{\Dst}{{}^{*\!}D}
\newcommand{\Fst}{{}^{*\!}F}
\newcommand{\Gst}{{}^{*}G}

\newcommand{\Rst}{{}^{*\!}R}

\newcommand{\Ust}{{}^{*}U}
\newcommand{\Vst}{{}^{*}V}
\newcommand{\Wst}{{}^{*}W}
\newcommand{\Xst}{{}^{*\!}X}
\newcommand{\Yst}{{}^{*}Y}
\newcommand{\as}{{}^{*\!}a}
\newcommand{\xst}{{}^{*\!}x}
\newcommand{\xo}{{}^{\circ\!}x}
\newcommand{\fst}{{}^{*\!}f}
\newcommand{\fo}{{}^{\circ\!}f}

\newcommand{\phist}{{}^{*\!}\phi}
\newcommand{\Fwtlst}{{}^{*\!}\widetilde F}
\newcommand{\Gwtlst}{{}^{*}\widetilde G}
\newcommand{\alfa}{\alpha}
\newcommand{\gama}{\gamma}

\newcommand{\eps}{\varepsilon}
\newcommand{\lam}{\lambda}
\newcommand{\sig}{\sigma}

\newcommand{\Gam}{\varGamma}

\newcommand{\Lam}{\varLambda}
\newcommand{\alfab}{\boldsymbol{\alpha}}
\newcommand{\alfabs}{{}^{*\!}\boldsymbol{\alpha}}
\newcommand{\betab}{\boldsymbol{\beta}}
\newcommand{\betabs}{{}^{*\!}\boldsymbol{\beta}}
\newcommand{\pib}{\boldsymbol{\pi}}
\newcommand{\phib}{\boldsymbol{\phi}}
\newcommand{\phibs}{{}^{*\!}\boldsymbol{\phi}}

\newcommand{\apr}{\approx}
\newcommand{\aprX}{\approx_X}
\newcommand{\aprY}{\approx_Y}
\newcommand{\sbs}{\subseteq}
\newcommand{\sps}{\supseteq}

\newcommand{\nin}{\notin}
\newcommand{\cd}{\cdot}
\newcommand{\cx}{\times}
\newcommand{\ox}{\otimes}
\newcommand{\co}{\circ}
\newcommand{\mto}{\mapsto}

\newcommand{\rst}{\!\restriction\!}

\newcommand{\wtl}{\widetilde}
\newcommand{\Imp}{\ \Rightarrow\ }
\newcommand{\Iff}{\ \Leftrightarrow\ }
\newcommand{\all}{\forall\,}
\renewcommand{\:}{\colon}
\begin{document}

\title[Stability of functional equations in the compact-open topology]%
{Stability of systems of general functional equations
in the compact-open topology}

\author[P. Zlato\v{s}]{Pavol Zlato\v{s}}

\address{%
Faculty of Mathematics, Physics and Informatics, Comenius University,
\newline\indent
Mlynsk\'a dolina, 842\,48~Bratislava, Slovakia}

\email{zlatos@fmph.uniba.sk}

\keywords{System of functional equations, continuous solution,
stability, locally compact, completely regular, uniformity,
nonstandard analysis}

\subjclass[2010]{Primary 39B82; \\ \indent
{}\phantom{1991 \ {\it Mathematics Subject Classification.}}
Secondary 39B72, 54D45, 54E15, 54J05}

\thanks{Research supported by the grant no.~1/0608/13
of the Slovak grant agency VEGA}

\begin{abstract}
We introduce a fairly general concept of functional equation for $k$-tuples
of functions $f_1,\dots,f_k\: X \to Y$ between arbitrary sets. The homomorphy
equations for mappings between groups and other algebraic systems, as well
as various types of functional equations and recursion formulas occurring in
mathematical analysis or combinatorics, respectively, become special cases
(of systems) of such equations. Assuming that $X$ is a locally compact and $Y$
is a completely regular topological space, we show that systems of such
functional equations, with parameters satisfying rather a modest continuity
condition, are stable in the following intuitive sense:
Every $k$-tuple of ``sufficiently continuous,'' ``reasonably bounded''
functions $X \to Y$ satisfying the given system with a ``sufficient precision''
on a ``big enough'' compact set is already ``arbitrarily close'' on an
``arbitrarily big'' compact set to a $k$-tuple of continuous functions solving
the system. The result is derived as a consequence of certain intuitively
appealing ``almost-near'' principle using the relation of infinitesimal
nearness formulated in terms of nonstandard analysis.
\end{abstract}

\maketitle


\noindent
The study of stability of functional equations in the spirit of Ulam started
with examining the stability of additive functions and more generally of
homomorphisms between metrizable topological groups, cf.~\cite{U1}, \cite{U2},
\cite{Ml}, \cite{HR}. Since that time it has developed to an established
full-value topic in mathematical and functional analysis and extended to an
extensive variety of (systems of) functional equations\,---\,see, e.g.,
\cite{Cz}, \cite{Fo}, \cite{HIR}, \cite{Ra1}, \cite{Ra2}, \cite{Sk}. However,
in most cases the stability issue was considered (explicitly or implicitly)
within the topology of uniform convergence or within the (strong) topology
given by a norm on some functional space. On the other hand, especially when
dealing with spaces of continuous functions defined on a locally compact space,
the compact-open topology (i.e., the topology of uniform convergence on compact
sets) is the most natural one. The systematic study of such local stability on
compacts and its relation to the ``usual'' global or uniform stability was
commenced by the author for homomorphisms between topological groups in
\cite{Z2} and extended to homomorphisms between topological universal algebras
in \cite{Z3}; cf.~also \cite{MZ}, \cite{SlZ}.

In the present paper we introduce a fairly general concept of functional
equation for $k$-tuples of functions $f_1,\dots,f_k\: X \to Y$ between
arbitrary sets. Then the homomorphy equations for mappings between groups
and other algebraic systems, as well as various types of functional equations
occurring in mathematical analysis (like, e.g., the sine and cosine addition
formulas) or various recursion formulas occurring in combinatorics become
just special cases (of systems) of such equations. Assuming that $X$ is a
locally compact and $Y$ is a completely regular (i.e., uniformizable)
topological space, we will show that systems of such functional equations,
with functional parameters satisfying rather a modest continuity condition,
are stable in the following intuitive sense, which will be made precise in
the final Section~4 (Theorems~\ref{thm4.4}, \ref{thm4.9}): Every $k$-tuple
of ``sufficiently continuous,'' ``reasonably bounded'' functions $X \to Y$
satisfying the given system with a ``sufficient precision'' on a ``big
enough'' compact set is already ``arbitrarily close'' on an ``arbitrarily
big'' compact set to a $k$-tuple of continuous functions solving the system.
The result is a generalization comprising several former results by the author
and his collaborators \cite{SlZ}, \cite{SpZ}, \cite{Z1}, \cite{Z2}, \cite{Z3}.
It is derived as a consequence of certain intuitively appealing stability or
``almost-near'' principle (in the sense of \cite{An}, \cite{BB}) using the
relation of infinitesimal nearness formulated in terms of nonstandard analysis
in Section~3 (Theorem~\ref{thm3.1}, Corollary~\ref{cor3.2}), generalizing a
more specific principle of this kind from \cite{SlZ}.

\section{General form of functional equations}\label{1}

\noindent
Let $X$, $Y$ be arbitrary nonempty sets and $k, m, n \ge 1$, $p \ge 0$ be
integers. A $k$-tuple of functions $\fb = (f_1,\dots,f_k)$, $f_i\: X \to Y$,
is viewed as a single function $\fb\: X \to Y^k$. Further, let
$\alfab = (\alfa_1,\dots,\alfa_m)$ be an $m$-tuple of $p$-ary operations
$\alfa_j\: X^p \to X$ (if $p=0$, a nullary operation $\alfa$ on $X$ is simply
a constant $\alfa \in X$). We use the tensor product notation to denote the
function $\fb \ox \alfab\: X^p \to Y^{k \cx m}$ assigning to every $p$-tuple
$\xb = (x_1,\dots,x_p) \in X^p$ the $k \cx m$ matrix
$$
(\fb \ox \alfab)(\xb) = \Bigl(f_i\bigl(\alfa_j(\xb)\bigr)\Bigr) =
\begin{pmatrix}
f_1\bigl(\alfa_1(\xb)\bigr) &\hdots & f_1\bigl(\alfa_m(\xb)\bigr) \\
\vdots &\ddots &\vdots \\
f_k\bigl(\alfa_1(\xb)\bigr) &\hdots & f_k\bigl(\alfa_m(\xb)\bigr)
\end{pmatrix}\,.
$$
In the trivial case when $k = m = 1$ we can identify $\fb = f$,
$\alfab = \alfa$; then $\fb \ox \alfab$ is just the composition of functions
$f \co \alfa\: X^p \to Y$. If $m=1$ and $\alfa(x) = x$ is the identity
$\Id_X$ on $X$, then $\fb \ox \alfa = (f_1,\dots,f_k) = \fb$. If $m=p$
and $\alfab = \pib = (\pi_1,\dots,\pi_m)$, where $\pi_j\:X^m \to X$
is the $j$th projection, i.e., $\pi_j(x_1,\dots,x_m) = x_j$, then
$(\fb \ox \pib)(\xb) = \bigl(f_i(x_j)\bigr) \in Y^{k \cx m}$. In general,
the function $\fb \ox \alfab$ can be identified with the matrix of
composed functions $f_i \co \alfa_j\: X^p \to Y$ ($i \le k$, $j \le m$).

Additionally, if \,$F\:Y^{k \cx m} \to Y$ is a $(k \cx m)$-ary operation
on $Y$, then $F(\fb \ox \alfab) = F \co (\fb \ox \alfab)\: X^p \to Y$ denotes
the function given by
$$
F(\fb \ox \alfab)(\xb) = F\bigl((\fb \ox \alfab)(\xb)\bigr)\,,
$$
for $\xb \in X^p$. More generally, for any mapping
$F\:Y^{k \cx m} \cx X^p \to Y$ we denote by
$\wtl F(\fb \ox \alfab)\: X^p \to Y$ the function given by
$$
\wtl F(\fb \ox \alfab)(\xb) = F\bigl((\fb \ox \alfab)(\xb),\xb\bigr)\,,
$$
for $\xb \in X^p$. Further on (except for some Examples) we will study
exclusively the latter more general case which includes the former one,
when the mapping $F$ does not depend on $\xb$, i.e., when
$F(\Ab,\ub) = F(\Ab,\vb)$ for any matrix $\Ab \in Y^{k \cx m}$ and
$\ub,\vb \in X^p$.

A \textit{general functional equation}, briefly a \textit{GFE},
\textit{of type} $(k,m,n,p)$, with \hbox{$k,m,n \ge 1$,} $p \ge 0$,
is a functional equation of the form
\begin{equation}\label{GFE0}
\wtl F(\fb \ox \alfab) = \wtl G(\fb \ox \betab)\,,
\end{equation}%
where $\fb = (f_1,\dots,f_k)$ is a $k$-tuple of
\textit{functional variables} or ``unknown'' functions $f_i\: X \to Y$,
$\alfab = (\alfa_1,\dots,\alfa_m)$ is an $m$-tuple and
$\betab = (\beta_1,\dots,\beta_n)$ is an $n$-tuple of $p$-ary
operations on the set $X$, and, finally,
$F\:Y^{k \cx m} \cx X^p \to Y$ and \hbox{$G\:Y^{k \cx n} \cx X^p \to Y$} are
any mappings. The operations (mappings) $\alfa_i$, $\beta_j$, $F$ and $G$ are
called the \textit{functional coefficients} or \textit{parameters} of the
equation. A~$k$-tuple of functions $\fb = (f_1,\dots,f_k)\: X \to Y^k$
\textit{satisfies} the GFE \eqref{GFE0}, or it is a \textit{solution}
of it, if the functions $\wtl F(\fb \ox \alfab)$, $\wtl G(\fb \ox \betab)$
coincide, i.e., if
$$
F\bigl((\fb \ox \alfab)(\xb),\xb\bigr) =
G\bigr((\fb \ox \betab)(\xb),\xb\bigr)\,,
$$
for all $\xb \in X^p$. More generally, $\fb$ \textit{satisfies} the GFE
\eqref{GFE0} \textit{on a set} $S \sbs X^p$ if the above equation holds for
each $\xb \in S$; we say that $\fb$ \textit{satisfies} the GFE \eqref{GFE0}
\textit{on a set} $A \sbs X$ if it satisfies \eqref{GFE0} on the set
$A^p \sbs X^p$.

A system of GFEs
\begin{equation}\label{GFElam}
\wtl F_\lam(\fb \ox \alfab_\lam) =
   \wtl G_\lam(\fb \ox \betab_\lam)
\qquad (\lam \in \Lam)\,,
\end{equation}
with (finite or infinite) index set $\Lam \ne \emptyset$, consist of GFEs
of particular types $(k,m_\lam,n_\lam,p_\lam)$ (with $k$ fixed and
$m_\lam$, $n_\lam$, $p_\lam$ depending on $\lam \in \Lam$). Then
$\fb = (f_1,\dots,f_k)$ is a \textit{solution of the system} if $\fb$
satisfies all the equations in it. Satisfaction of the system on some set
$A \sbs X$ is defined in the obvious way.

We do not maintain that the (systems of) GFEs of the form just defined cover
all the (systems of) functional equations one can meet, as such a claim would
be too ambitious and, obviously, not founded well enough. In particular,
functional equations dealing with compositions of functional variables
$f_i \co f_j$ or with iterated compositions like $f$, $f^2 = f \co f$,
$f^3 = f \co f \co f$, etc., do not fall under this scheme. On the other
hand, as indicated by the examples below, they still comprise a large and
representative variety of (systems of) functional equations studied so far.

Let us start with three closely related examples of algebraic nature.

\begin{ex}\label{Exhomgrp}
Let $(X,*)$, $(Y,\star)$ be two grupoids, i.e., algebraic structures with
arbitrary binary operations $*$, $\star$ on the sets $X$ and $Y$, respectively.
Let $\alfa\:X^2 \to X$ be the operation $\alfa(x_1,x_2) = x_1 * x_2$ on $X$,
$\pi_1,\,\pi_2\:X^2 \to X$ be the projections on the first and the second
variable, respectively. Further, let $F = \Id_Y$ be the identity map on $Y$
and $G(y_1,y_2) = y_1 \star y_2$ be the operation on $Y$. Then the GFE
$$
F(\fb \ox \alfab) = G(\fb \ox \pib)
$$
of type $(1,1,2,2)$, with $\fb = f\:X \to Y$, $\alfab = \alfa$ and
$\pib = (\pi_1,\pi_2)$, which rewrites as
$$
f \co \alfa = G\bigl(f \ox (\pi_1,\pi_2)\bigr)\,,
$$
simply means that
$$
f(x_1 * x_2) = f(x_1) \star f(x_2)
$$
for all $x_2,x_2 \in X$. In other words, a function $f$ satisfies the above
GFE if and only if it is a homomorphism $f\:(X,*) \to (Y,\star)$.

If both $(X,*)$, $(Y,\star)$ coincide with the additive group $(\bR,+)$
of reals, we get the Cauchy functional equation
$$
f(x + y) = f(x) + f(y)\,.
$$
If $(X,*) = (\bR,+)$ and $(Y,\star)$ is the multiplicative group
$(\bR^+,\cd)$ of positive reals, we obtain the equation
$$
f(x + y) = f(x) f(y)\,,
$$
characterizing exponential functions. If both $(X,*)$, $(Y,\star)$ denote the
set $\bR$ with the arithmetical mean $x * y = x \star y = (x+y)/2$, we have
Jensen's functional equation
$$
f\left(\frac{x+y}{2}\right) = \frac{f(x) + f(y)}{2}\,.
$$
And the list could be continued indefinitely.
\end{ex}

\begin{ex}\label{ExhomUA}
More generally, let $\Lam$ be a set of operation symbols with finite arities
$p_\lam$ ($\lam \in \Lam$), and $\XX = (X,\alfa_\lam)_{\lam \in \Lam}$
be $\YY = (Y,G_\lam)_{\lam \in \Lam}$ be two universal algebras of
signature $(p_\lam)_{\lam \in \Lam}$, i.e.,
$\alfa_\lam = \lam^\XX\: X^{p_\lam} \to X$,
$G_\lam = \lam^\YY\: Y^{p_\lam} \to Y$ are $p_\lam$-ary operations on the
sets $X$, $Y$, respectively, corresponding to the symbol $\lam \in \Lam$,
cf.~\cite{G}. A function $f\:X \to Y$ is called a \textit{homomorphism} from
$\XX$ to $\YY$, briefly $f\:\XX \to \YY$, if for each $\lam \in \Lam$ and any
$p_\lam$-tuple $\xb = (x_1,\dots,x_{p_\lam}) \in X^{p_\lam}$ we have
$$
f\bigl(\alfa_\lam(x_1,\dots,x_{p_\lam})\bigr)
= G_\lam\bigl(f(x_1),\dots,f(x_{p_\lam})\bigr)\,,
$$
(for nullary operation symbols $\lam \in \Lam$ this simply means that
$f(\alfa_\lam) = G_\lam$). Similarly as in the previous Example~\ref{Exhomgrp},
we see immediately that this is the case if and only if $f$ satisfies the
system of GFEs
$$
f \co \alfa_\lam = G_\lam\bigl(f \ox (\pi_1,\dots,\pi_{p_\lam})\bigr)
\qquad (\lam \in \Lam)\,,
$$
of types $(1,1,p_\lam,p_\lam)$, where $\pi_j\:X^{p_\lam} \to X$,
$\pi_j(\xb) = x_j$, is the $j$th projection for $j \le p_\lam$.
\end{ex}

\begin{ex}\label{Exmodmorf}
Let $(\Lam,+,\cd,0,1)$, be a ring with unit $1 \ne 0$. A (left) $\Lam$-module
$X$ is an abelian group $(X,+)$ with scalar multiplication $\Lam \cx X \to X$,
sending each pair $(\lam,x) \in \Lam \cx X$ to the scalar multiple
$\lam x \in X$, satisfying the usual axioms. Then each scalar $\lam \in \Lam$
can be regarded as an endomorphism $\lam^X\: X \to X$ of
the abelian group $(X,+)$, and the assignment $\lam \mto \lam^X$ becomes
a homomorphisms of rings
$(\Lam,+,\cd,0,1) \to \bigl(\End(X,+), +, \co, 0, \Id_X\bigr)$, cf.~\cite{HS}.
In particular, if $\Lam$ is a field, then a $\Lam$-module is just a vector
space over $\Lam$.

A homomorphism of $\Lam$-modules $X$, $Y$ is a mapping $f\: X \to Y$,
preserving the addition and scalar multiplication, i.e., satisfying
$$\begin{aligned}
f(x + y) &= f(x) + f(y)\,,\\
f(\lam x) &= \lam f(x)
\end{aligned}$$
for any $x,y \in X$, $\lam \in \Lam$. If $\Lam$ is a field, then this is the
usual definition of a linear mapping between the vector spaces $X$, $Y$.

Regarding $\Lam^+ = \{+\} \cup \Lam$ as a set of operation symbols (+~binary,
and each $\lam \in \Lam$ unary), every $\Lam$-module is simply a universal
algebra $\XX = (X,+,\lam)_{\lam \in\Lam}$, satisfying the $\Lam$-module axioms,
and a $\Lam$-module homomorphism is a homomorphism of such algebras. Now, the
previous Example~\ref{ExhomUA} applies, i.e., $f\: X \to Y$ is a $\Lam$-module
homomorphism if and only if it satisfies the system of GFEs consisting of
$$
f \co \alfa = G\bigl(f \ox (\pi_1,\pi_2)\bigr)\,,
$$
where $\alfa$ is the addition in $X$ and $G$ is the addition in $Y$, and
$$
f \co \lam^X = \lam^Y \co f \qquad (\lam \in \Lam)\,.
$$
\end{ex}

We continue with two examples of more analytic character.

\begin{ex}\label{Exsincos}
Let $\alfa\: \bR^2 \to \bR$ be the addition on $\bR$,
$F_1 = \pi_1,\,F_2 = \pi_2\: \bR^2 \to \bR$ denote the projections,
and the functions $G_1,G_2\: \bR^{2 \cx 2} \to \bR$ be given by
$$\begin{aligned}
G_1\begin{pmatrix}
a_{11} & a_{12}\\ a_{21} & a_{22}\end{pmatrix}
&= \Per\begin{pmatrix}
a_{11} & a_{12}\\ a_{21} & a_{22}
\end{pmatrix} =
a_{11} a_{22} + a_{21} a_{12}\,,\\
G_2\begin{pmatrix}
a_{11} & a_{12}\\ a_{21} & a_{22}\end{pmatrix}
&= \Det\begin{pmatrix}
a_{21} & a_{12}\\ a_{11} & a_{22}\end{pmatrix}
= a_{21} a_{22} - a_{11} a_{12}\,,
\end{aligned}$$
(notice the reversed order of elements in the first column of the determinant).
Then the system of the following two GFEs, both of type $(2,1,2,2)$, in the
couple of functional variables $\fb = (f_1,f_2)$, standing for the sine and
cosine, respectively,
$$\begin{aligned}
\pi_1(\fb \ox \alfa) = G_1\bigl(\fb \ox (\pi_1,\pi_2)\bigr)\,,\\
\pi_2(\fb \ox \alfa) = G_2\bigl(\fb \ox (\pi_1,\pi_2)\bigr)\,,
\end{aligned}$$
is nothing else but the well known sine and cosine addition formulas
$$\begin{aligned}
\sin(x + y) &= \sin x \cos y + \cos x \sin y\,, \\
\cos(x + y) &= \cos x \cos y - \sin x \sin y\,.
\end{aligned}$$
\end{ex}

\begin{ex}\label{ExGamma}
Let $\sig\: \bC \to \bC$ denote the shift $\sig(x) = x + 1$ and
$G\:\bC \cx \bC \to \bC$ be the multiplication $G(y,x) = yx$ on $\bC$.
Then the GFE
$$
f \co \sig = \wtl G(f)\,
$$
of type $(1,1,1,1)$ is the functional equation
$$
f(x+1) = f(x)\,x\,,
$$
satisfied by Euler's function $\Gamma$ on the open complex halfplane
$\{x \in \bC\mid \Re x > 0\}$.
\end{ex}

We conclude with two examples dealing with recursion in one and two variables.

\begin{ex}\label{Exrecurs}
Let $\bigl(f(x)\bigr)_{x\in\bN}$ be a sequence of elements of a set $A$, i.e.,
a function $f\:\bN \to A$, satisfying the recursion
$$
f(x + n) = G\bigl(f(x),\dots,f(x+n-1)\bigr)
$$
for a fixed $n \ge 1$ and an $n$-ary operation $G\: A^n \to A$ given in
advance. For each $j \in \bN$ we denote by $\sig^j\: \bN \to \bN$ the shift
$\sig^j(x) = x + j$. Then the above recursion formula takes the form of the
GFE
$$
f \co \sig^n = G\bigl(f \ox \bigl(\sig^0,\dots,\sig^{n-1}\bigr)\bigr)
$$
of type $(1,1,n,1)$. The more general recursion formula
$$
f(x + n) = G\bigl(f(x),\dots,f(x+n-1),x\bigr)\,,
$$
where $G\: A^n \cx \bN \to A$, takes the form of the GFE
$$
f \co \sig^n =
\wtl G\bigl(f \ox \bigl(\sig^0,\dots,\sig^{n-1}\bigr)\bigr)
$$
of type $(1,1,n,1)$.
\end{ex}

\begin{ex}\label{ExRecMult}
Let $A$ be a set and $G\:A^3 \cx \bN^2 \to A$ be an arbitrary mapping.
Consider the following recursion formula
$$
f(x+1,y+1) = G\bigl(f(x,y),f(x+1,y),f(x,y+1),x,y\bigr)\,,
$$
expressing the value of a function $f\:\bN^2 \to A$ at $(x+1,y+1)$ in terms
of its values at the preceding neighbors $(x,y)$, $(x+1,y)$, $(x,y+1)$, and
the position $(x,y)$ itself. The notorious recursion formulas
$$\begin{aligned}
\binom{n+1}{k+1} &= \binom{n}{k} + \binom{n}{k+1}\,,\\
c(k+1,l+1) &= c(k+1,l) + c(k,l+1)\,, \\
s(n+1,k+1) &= s(n,k) - n\,s(n,k+1)\,,\\
S(n+1,k+1) &= S(n,k) + (k+1)S(n,k+1)\,,
\end{aligned}$$
for binomial coefficients (both in the usual form or for
$c(k,l) = \binom{k+l}{k}$), as well as for Stirling numbers of the first and
the second kind, respectively, are just some special cases of such functional
equations for functions $f\: \bN^2 \to \bZ$.

Let $\sig_1,\sig_2\:\bN^2 \to \bN^2$ denote the shifts in the first and the
second variable, respectively, i.e., $\sig_1(x,y) = (x+1,y)$,
$\sig_2(x,y) = (x,y+1)$, and
$\sig_{12} = \sig_1 \co \sig_2 = \sig_2 \co \sig_1\:\bN^2 \to \bN^2$ be the
double shift, i.e., $\sig_{12}(x,y) = (x+1,y+1)$. Then the original recursion
formula can be written as the GFE
$$
f \co \sig_{12} = \wtl G\bigl(f \ox (\sig_0,\sig_1,\sig_2)\bigr)\,,
$$
with $\sig_0 = \Id_{\bN^2}\:\bN^2 \to \bN^2$ denoting the identity.
The generalization to recursion formulas for functions $f\:\bN^n \to A$
with $n \ge 2$ variables is straightforward.
\end{ex}

\section{Infinitesimal nearness and $S$-continuity}\label{2}

\noindent
In this section we modify the short introduction to the nonstandard approach to
continuity of mappings between topological groups from \cite{SlZ} to the more
general situation of mappings between completely regular topological spaces.
We use \cite{E} as a reference source for general topology. In order to
simplify our terminology, we assume that all (standard) topological or uniform
spaces dealt with are Hausdorff.

The reader is assumed to have some basic acquaintance with nonstandard
analysis in an extent covered either by the original Robinson's monograph
\cite{R} or, e.g., by \cite{AFHL}, or \cite{D}, or \cite{ACH}, mainly in
the parts \cite{H} and \cite{L}. In particular, some knowledge of the
nonstandard approach to topology, based on the equivalence relation of
infinitesimal nearness, is desirable.

Our exposition takes place in a~nonstandard universe which is an elementary
extension $\Vbs$ of a superstructure $\Vb$ over some set of individuals
containing at least all (classical) complex numbers and the elements of the
universal algebras or topological space dealt with. In particular, this means
that every standard universal algebra $\AAf = (A,F_\lam)_{\lam\in\Lam}$
is embedded into its nonstandard extension
$\AAst = (\Ast,\Fst_\lam)_{\lam\in\Lam}$ via the canonic elementary embedding
$a \mto \as$, and identified with its image under~${}^*$, in such a way that
for any formula $\Phi(v_1,\dots,v_n)$ of the first-order language built upon
the operation symbols $\lam \in \Lam$ and any $a_1,\dots,a_n \in A$ we have:
$$
\Phi(a_1,\dots,a_n) \ \text{holds in $\AAf$} \quad \text{if and only if}
\quad {}^*\Phi(a_1,\dots,a_n) \ \text{holds in $\AAst$},
$$
where ${}^*\Phi$ is the formula obtained from $\Phi$ by replacing each
operation $F_\lam\: A^{p_\lam} \to A$ by its extension
$\Fst_\lam\: \Ast^{p_\lam} \to \Ast$. This rule is referred to as the
\textit{transfer principle}. However, this principle applies to any tuples of
functions $\fb = (f_1,\dots,f_k)\: X \to Y^k$ and their nonstandard extensions
$\fbs = (\fst_1,\dots,\fst_k)\: \Xst \to \Yst^k$, as well.

Objects belonging to the original universe are called \textit{standard} and
objects belonging to its nonstandard extension are called \textit{internal}.
Taking the advantage of the relation between the universes of standard and
internal objects, we cannot avoid the so called \textit{external sets}, i.e.,
sets of internal objects, which themselves are not internal.

We assume that our nonstandard universe is $\kappa$-\textit{saturated} for
some sufficiently big uncountable cardinal $\kappa$, which will be specified
latter on. This is to say that any system of less than $\kappa$ internal sets
with the finite intersection property has itself nonempty intersection.
Informally, we refer to this situation by the phrase that our nonstandard
universe is \textit{sufficiently saturated}. In a similar vein, a set of
\textit{admissible size} means a set of cardinality~$< \kappa$.

If $(X,\T)$ is a topological space, then the topology $\T$ (i.e., the system
of open sets in $X$) gives rise to two different topologies on its nonstandard
extension $\Xst$.

The {\it $Q$-topology\/} is given by the base ${}^*\T$; it is Hausdorff if and
only if the original topology $\T$ on $X$ is. This topology plays rather an
auxiliary role in our accounts.

The {\it $S$-topology\/} is given by the base $\{\Ast;\,A \in \T\}$. Obviously,
the $S$-topology is coarser than the $Q$-topology and it is not Hausdorff,
unless $(X,\T)$ is discrete.

We will systematically take advantage of the fact that if $(X,\T)$ is a
(Hausdorff) completely regular space, whose topology is induced by a uniformity
$\U$ on $X$, then, in a sufficiently saturated nonstandard universe, the
$S$-topology is fully determined by a single external equivalence relation
$$
x \apr y \Iff \all U \in \U : (x,y) \in \Ust\,,
$$
called the \textit{relation of infinitesimal nearness} on $\Xst$.
At the same time the system $\{\Ust\mid U \in \U\}$ is a base of the
\textit{$S$-uniformity} on $\Xst$. Uniform continuity with respect to it
is referred to as the \textit{uniform $S$-continuity}.

The external set of all elements indiscernible from $x \in \Xst$ is called
the \textit{monad} of $x$, i.e.,
$$
\Mon(x) = \{y \in \Xst \mid y \apr x\}\,.
$$
An element $x \in \Xst$ is called \textit{nearstandard} if $x \apr x_0$ for
some $x_0 \in X$. The (external) set of all nearstandard elements in $\Xst$
is denoted by $\Ns(\Xst)$, i.e.,
$$
\Ns(\Xst) = \bigcup_{x \in X} \Mon(x)\,.
$$
For $x \in \Ns(\Xst)$ we denote by $\xst$ the the unique element $x_0 \in X$
infinitesimally close to~$x$, called the \textit{standard part} or
\textit{shadow} of~$x$.

For the rest of this section  $(X,\T_X)$ and $(Y,\T_Y)$ denote some completely
regular topological spaces, whose topologies are induced by some uniformities
$\U_X$, $\U_Y$, respectively, and $\Xst$, $\Yst$ are their canonical extensions
in a sufficiently saturated nonstandard universe; more precisely, we assume
that our nonstandard universe is $\kappa$-saturated for some cardinal $\kappa$
bigger than the cardinalities of some bases of the uniformities $\U_X$, $\U_Y$.

While the \,$Q$-continuity of internal functions $f\:\Xst \to \Yst$ is just the
${}^{*}$continuity, their $S$-continuity can be characterized in the following
intuitively appealing way in the spirit of the original infinitesimal calculus
(below, we denote the relations of infinitesimal nearness on $\Xst$, $\Yst$ by
$\aprX$, $\aprY$, respectively):

\begin{prop}\label{prop2.1}
Let $f\:\Xst \to \Yst$ be an internal function. Then
\begin{enumerate}
\item[(a)]
$f$ is $S$-continuous in a point $x_0 \in \Xst$ if and only if
$$
\all x \in \Xst: x \aprX x_0 \Imp f(x) \aprY f(x_0)\,;
$$
\item[(b)]
$f$ is $S$-continuous on a set $A \sbs \Xst$ (i.e., $f$ is $S$-continuous
in every point $a \in A$) if and only if
$$
\all a \in A \,\ \all x \in \Xst : x \aprX y \Imp f(x) \aprY f(y)\,;
$$
\item[(c)]
if $A \sbs \Xst$ is an intersection of admissibly many internal sets,
then $f$ is $S$-continuous on $A$ if and only if $f$ is uniformly
$S$-continuous on $A$.
\end{enumerate}
\end{prop}

In view of (a) and (b), $S$-continuity of an internal function
$f\:\Xst \to \Yst$ can be alternatively \textit{defined} as preservation of
the relation of infinitesimal nearness by~$f$. In particular, for the canonic
extension $\fst\:\Xst \to \Yst$ of a standard function $f\:X \to Y$ we have
the following criteria (notice the subtle difference between (b) and (c)).

\begin{cor}\label{cor2.2}
Let $f\: X \to Y$ be a function. Then
\begin{enumerate}
\item[(a)]
$f$ is continuous in a point $x_0 \in X$ if and only if
$$
\all x \in \Xst : x \aprX x_0
                    \Imp \fst(x) \aprY f(x_0)\,;
$$
\item[(b)]
$f$ is continuous on a set $A \sbs X$ (i.e., $f$ is continuous in every
point $a \in A$) if and only if
$$
\all a \in A\,\ \all x \in \Xst : x \aprX a \Imp \fst(x) \aprY f(a)\,;
$$
\item[(c)]
$f$ is uniformly continuous on a set $A \sbs X$ if and only if
$$
\all x,y \in \Ast : x \aprX y \Imp \fst(x) \aprY \fst(y)\,.
$$
\end{enumerate}
\end{cor}

Notice that under the assumption of (b), $\fst$ is $Q$-continuous on $\Ast$,
as well.
\smallskip

An internal function $f\: \Xst \to \Yst$ is called \textit{nearstandard}
if $f(x) \in \Ns(\Yst)$ for each $x \in X$. Let us remark that this is indeed
equivalent to $f$ be a nearstandard point in the nonstandard extension
${}^*\bigl(Y^X\bigr)$ of the Tikhonov product $Y^X = \{f \mid f\: X \to Y\}$.
Any nearstandard function $f\: \Xst \to \Yst$ gives rise to a function
$\fo\: X \to Y$ given by
$$
(\fo)(x) = {}^\co(f(x))\,,
$$
for $x \in X$, called the \textit{standard part} of $f$. If $f$ is additionally
$S$-continuous on $\Ns(\Xst)$, then its standard part can be extended to a map
$\fo\: \Ns(\Xst) \to Y$ (denoted in the same way), such that
$$
\fo(x) = \fo(\xo) = {}^{\co\!}(f(x))
$$
for any $x \in \Ns(\Xst)$. The situation can be depicted by the following
commutative diagram:
$$\begin{CD}
\Ns(\Xst) @>f>> \Ns(\Yst) \\
@V{{}\co}VV @VV{{}\co}V \\
X @>>{\fo}> Y
\end{CD}$$
\smallskip

A function $f\: \Xst \to \Yst$ is called \textit{$N\!S$-continuous} if it is
$S$-continuous on $\Ns(\Xst)$. Now we have the following supplement to
Proposition~\ref{prop2.1} and its Corollary~\ref{cor2.2}.

\begin{prop}\label{prop2.3}
Let $f\: \Xst \to \Yst$ be a nearstandard internal function. Then the
following implications hold:
\begin{enumerate}
\item[(a)]
if $f$ is $N\!S$-continuous, then its standard part
\,$\fo\: X \to Y$ is continuous and \,${}^*(\fo)(x) \aprY f(x)$ for
$x \in \Ns(\Xst)$;
\item[(b)]
if $f$ is $S$-continuous on some internal set $A \sps \Ns(\Xst)$, then
its standard part \,$\fo\: X \to Y$ is uniformly continuous.
\end{enumerate}
\end{prop}

Notice that the function \,${}^*(\fo)$ is also $Q$-continuous .
However, even if $f$ were $S$-continuous on the whole of \,$\Xst$, the second
conclusion in (a) still cannot be strengthened to \,${}^*(\fo)(x) \aprY f(x)$
for all $x \in \Xst$.

\begin{proof}
We will prove just the first statement in (a); then the second statement
easily follows and (b) can be proved in a similar way.

Assume that $f$ is $N\!S$-continuous and denote $g = \fo\: X \to Y$ its
standard part. In order to prove the continuity of $g$, pick arbitrary
$x_0 \in X$ and $V \in \U_Y$. Let $W \in \U_Y$ be symmetric, such that
$W^3 = W \co W \co W \sbs V$. As $f$ is internal and $N\!S$-continuous,
it is also continuous in $x_0$ with respect to the $S$-topology on $\Xst$,
hence there is a $U \in \U_X$ such that
$(x,x_0) \in \Ust$ implies $\bigl(f(x),f(x_0)\bigr) \in \Wst$ for any
$x \in \Xst$. In particular, for $x \in X$ such that $(x,x_0) \in U$,
we have $g(x) \aprY f(x)$, $\bigl(f(x),f(x_0)\bigr) \in \Ust$ and
$f(x_0) \aprY g(x_0)$, hence $\bigl(g(x),g(x_0)\bigr) \in \Wst^3 \sbs \Vst$.
As $g(x), g(x_0) \in Y$, by transfer principle $\bigl(g(x),g(x_0)\bigr) \in V$.
\end{proof}

Let us conclude this section with a remark that the introduced continuity
notions can be easily generalized to tuples of functions
$\fb = (f_1,\dots,f_k)\: X \to Y^k$. As well known, $\fb$ has whatever standard
continuity property if and only if all the functions $f_i$ have this property.
The relation of infinitesimal nearness $\aprY$ can be extended to $\Yst^k$ by
$$
\yb \aprY \zb \Iff y_1 \aprY z_1 \ \&\ \dots \ \& \ y_k \aprY z_k\,,
$$
(similarly, $\aprX$ can be extended to $\Xst^p$). Then an internal function
$\fb\: \Xst \to \Yst^k$ is nearstandard if and only if all the functions $f_i$
are nearstandard; $\fb$ has anyone of the $S$-continuity properties if and only
if all the functions $f_i$ have the corresponding property. If $\fb$ is
nearstandard, then the $k$-tuple $\fbo = \bigl(\fo_1,\dots,\fo_k\bigr)$ of
functions $\fo_i\: X \to Y$ is called the \textit{standard part} of $\fb$.

\section{An infinitesimal ``almost-near'' principle for \\
systems of general functional equations}\label{3}

\noindent
Let $(X,\T_X)$, $(Y,\T_Y)$ be two completely regular topological spaces with
topologies induced by some uniformities $\U_X$, $\U_Y$, respectively. If there
is no danger of confusion, we omit the subscripts $X$, $Y$ in the notation of
the relations of infinitesimal nearness $\aprX$, $\aprY$ on $\Xst$, $\Yst$,
respectively.

Consider the GFE~\eqref{GFE0} for a $k$-tuple of functional variables
$\fb = (f_1,\dots,f_k)$. Embedding the situation into some nonstandard universe
we say that an internal function $\fb = (f_1,\dots,f_k)\: \Xst \to \Yst^k$,
\textit{almost satisfies} the equation~\eqref{GFE0} \textit{on} $\Ns(\Xst)$ if
$$
\Fwtlst(\fb \ox \alfabs)(\xb) \apr \Gwtlst(\fb \ox \betabs)(\xb)
$$
for all $\xb = (x_1\dots,x_p) \in \Ns(\Xst^p)$. Similarly,
$\fb$ \textit{almost satisfies} the system of GFEs~\eqref{GFElam} \textit{on}
$\Ns(\Xst)$ if it almost satisfies on $\Ns(\Xst)$ every equation in it.
(Notice that, due to the transfer principle,
${}^{*\!}\bigl(\wtl F\bigr) = \wtl{\Fst}$, and similarly for $G$, hence the
notation $\Fwtlst$, $\Gwtlst$ is unambiguous.)

\begin{thm}\label{thm3.1}
Let the mappings $F\: Y^{k \cx m} \cx X^p \to Y$,
\,$G\: Y^{k \cx n} \cx X^p \to Y$ be continuous in the ``matrix'' variables
$y_{ij} \in Y$ for $i \le k$ and all $j \le m,n$, respectively. If a
nearstandard internal function $\fb =(f_1,\dots,f_k)\: \Xst \to \Yst^k$ almost
satisfies the GFE~\eqref{GFE0} on $\Ns(\Xst)$, then its standard part
$\fbo = \bigl(\fo_1,\dots,\fo_k\bigr)$ is a solution of the GFE~\eqref{GFE0}.
\end{thm}

\begin{proof}
Take an arbitrary $\xb = (x_1,\dots,x_p) \in X^p$. We have
$$
\Fwtlst(\fb \ox \alfabs)(\xb) \apr
\Gwtlst(\fb \ox \betabs)(\xb)\,.
$$
As $\xb$ is standard, $\alfabs(\xb) = \alfab(\xb)$ is standard, as well,
hence
$f_i\bigl(\alfa_j(\xb)\bigr) \apr \fo_i\bigl(\alfa_j(\xb)\bigr)$ for any
$i \le k$, $j \le m$, and, as $\Fst$ is $N\!S$-continuous in the matrix
variables $y_{ij}$,
$$\begin{aligned}
\Fwtlst(\fb \ox \alfabs)(\xb)
&= \Fst\bigl((\fb \ox \alfabs)(\xb), \xb\bigr)
= \Fst\bigl((\fb \ox \alfab)(\xb), \xb\bigr)\\
&\apr \Fst\bigl((\fbo \ox \alfab)(\xb), \xb\bigr)
= F\bigl((\fbo \ox \alfab)(\xb), \xb\bigr)
= \wtl F(\fbo \ox \alfab)(\xb)\,.
\end{aligned}$$
Similarly we can get
$$
\Gwtlst(\fb \ox \betabs)(\xb) \apr \wtl G(\fbo \ox \betab)(\xb)\,.
$$
Therefore,
$$
\wtl F(\fbo \ox \alfab)(\xb) \apr \wtl G(\fbo \ox \betab)(\xb)\,,
$$
and, as both the expressions are standard,
$$
\wtl F(\fbo \ox \alfab)(\xb) = \wtl G(\fbo \ox \betab)(\xb)\,,
$$
i.e., $\fbo$ is a solution of the GFE~\eqref{GFE0}.
\end{proof}

From Theorem~\ref{thm3.1} and Proposition~\ref{prop2.3}\,(b) we readily
obtain the following consequence generalizing Theorem~2.2 from \cite{SlZ},
dealing just with the homomorphy equation in topological groups.

\begin{cor}\label{cor3.2}
Assume that the mappings $F$, $G$ are continuous in the matrix variables
$y_{ij}$. Then for every nearstandard $N\!S$-continuous internal function
$\fb = (f_1,\dots,f_k)\: \Xst \to \Yst^k$ which almost satisfies the system
of GFEs~\eqref{GFElam} on $\Ns(\Xst)$, there is a continuous solution
$\phib = (\phi_1,\dots,\phi_k)$ of the system, such that
$\phib(x) \apr \fb(x)$ for each $x \in X$.
\end{cor}

\section{Stability of systems of general functional equations}\label{4}

\noindent
In order to formulate a standard version of the just established nonstandard
stability principle, we need to introduce some notions\,---\,cf.~\cite{SlZ},
\cite{Z2}, \cite{Z3}.

\begin{defn}\label{def4.1}
Let $(X,\T_X)$ be a topological space and $(Y,\U_Y)$ be a uniform space.

\begin{enumerate}
\item[(a)]
A \textit{$(\T_X,\U_Y)$ continuity scale} is a mapping
$\Gam\: X \cx \B \to \T_X$, such that $\B$ is a base of the uniformity
$\U_Y$ and $\Gam(x,V)$ is a neighborhood of $x$ in $(X,\T_X)$, satisfying
$$
V \sbs W \Imp \Gam(x,V) \sbs \Gam(x,W)
$$
for any $x \in X$, and $V,W \in \B$.
\item[(b)]
Given a continuity scale $\Gam\: X \cx \B \to \T_X$, a~function $f\: X \to Y$
is called \textit{$\Gam$-continuous in a point} $x_0 \in X$, or
\textit{continuous in $x_0$ with respect to $\Gam$}, if
$$
x \in \Gam(x_0,V) \Imp \bigl(f(x),f(x_0)\bigr) \in V
$$
for each $x \in X$; $f$ is \textit{$\Gam$-continuous on a set} $A \sbs X$ if it
is $\Gam$-continuous in each point $a \in A$; it is \textit{$\Gam$-continuous}
if it is $\Gam$-continuous on $X$.
\item[(c)]
Given a continuity scale $\Gam\: X \cx \B \to \T_X$ and an entourage
$U \in \B$, a~function $f\: X \to Y$ is \textit{$(\Gam,U)$-precontinuous
in a point} $x_0 \in X$ if
$$
x \in \Gam(x_0,V) \Imp \bigl(f(x),f(x_0)\bigr) \in V
$$
for any $V \in \B$, such that $U \sbs V$, and each $x \in X$;
$(\Gam,U)$-precontinuity on a set $A \sbs X$ and on $X$ are defined in the
obvious way.
\item[(d)]
If $(X,\U_X)$ is a uniform space, too, then a \textit{$(\U_X,\U_Y)$ uniform
continuity scale} is a mapping $\Gam\: \B \to \U_X$ such that $\B$ is some base
of the uniformity $\U_Y$ and
$$
V \sbs W \Imp \Gam(V) \sbs \Gam(W)
$$
for any $V,W \in \B$.
\item[(e)]
Given a uniform continuity scale $\Gam\: \B \to \U_X$, a~function $f\: X \to Y$
is \textit{uniformly $\Gam$-continuous on a set} $A \sbs X$ if
$$
(x,y) \in \Gam(V) \Imp \bigl(f(x),f(y)\bigr) \in V
$$
for any $x,y \in A$; $f$ is \textit{uniformly $\Gam$-continuous} if it is
uniformly $\Gam$-continuous on $X$.
\item[(f)]
Given a uniform continuity scale $\Gam\: \B \to \U_X$ and an entourage
$U \in \B$, a~function $f\: X \to Y$ is \textit{uniformly
$(\Gam,U)$-precontinuous on a set} $A \sbs X$ if
$$
(x,y) \in \Gam(V) \Imp \bigl(f(x),f(y)\bigr) \in V
$$
for for any $V \in \B$, such that $U \sbs V$ and all $x,y \in A$;
$f$ is \textit{uniformly $(\Gam,U)$-precontinuous} if it is
$(\Gam,U)$-precontinuous on $X$.
\end{enumerate}
\end{defn}

Obviously, if a function $f\: X \to Y$ is $\Gam$-continuous with respect to
some continuity scale $\Gam$, then it is continuous. Conversely, if $f$ is
continuous, then, given any base $\B$ of $\U_Y$, the assignment
$$
\Gam(x_0,V) = \bigl\{x \in X\ \big|\ \bigl(f(x),f(x_0)\bigr) \in V\bigr\}\,,
$$
for $x_0 \in X$, $V \in \B$, defines a $(\T_X,\U_Y)$ continuity scale
$\Gam\: X \cx \B \to \T_X$, and, of course, $f$ is continuous with respect
to it.

The other way round, $f$ is $\Gam$-continuous if and only if it is
$(\Gam,U)$-precontinuous for all $U \in \B$. Thus each particular condition
of $(\Gam,U)$-precontinuity for an entourage $U \in \U_Y$ can be regarded
as an approximate continuity property. Informally, $f$ is ``almost
$\Gam$-continuous'' if it is $(\Gam,U)$-precontinuous for a ``sufficiently
small'' $U \in \B$. The relation between the uniform versions of these notions
is similar.

If $(X,d)$, $(Y,e)$ are metric spaces, then a $(d,e)$-continuity scale is just
a mapping $\gama\: X \cx \bR^+ \to \bR^+$ such that
$\gama(x,\eps) \le \gama(x,\eps')$ for any $x \in X$, $\eps' \ge \eps > 0$.
Then a function $f\: X \to Y$ is $\gama$-continuous in $x_0 \in X$ if
$$
d(x,x_0) < \gama(x_0,\eps) \Imp e\bigl(f(x), f(x_0)\bigr) < \eps
$$
for all $\eps > 0$ and $x \in X$. A~uniform $(d,e)$-continuity scale is an
isotone mapping  $\gama\: \bR^+ \to \bR^+$. A function $f\: X \to Y$ is
uniformly $\gama$-continuous if, for all $\eps > 0$ and $x,y \in X$,
$$
d(x,y) < \gama(\eps) \Imp e\bigl(f(x), f(y)\bigr) < \eps\,.
$$

\begin{defn}\label{def4.2}
Let $X$, $Y$ be arbitrary sets.
\begin{enumerate}
\item[(a)]
A \textit{bounding relation} from $X$ to $Y$ is any binary relation
$R \sbs X \cx Y$ such that all its stalks
$R[x] = \{y \in Y \mid (x,y) \in R\}$, for $x \in X$, are nonempty.
\item[(b)]
Given a bounding relation $R \sbs X \cx Y$, a function $f\: X \to Y$ is
\textit{$R$-bounded on a set} $A \sbs X$ if $f(a) \in R[a]$ for each
$a \in A$; $f$~is \textit{$R$-bounded} if it is $R$-bounded on $X$, i.e.,
if $f \sbs R$.
\item[(c)]
A bounding relation $R \sbs X \cx Y$ is \textit{stalkwise finite} if all
its stalks $R[x]$ are finite. If, additionally, $(Y,\T_Y)$ is a topological
space, then $R$ is called \textit{stalkwise compact} if all its stalks $R[x]$
are compact.
\end{enumerate}
\end{defn}

\begin{defn}\label{def4.3}
Let $X$ be any set, $(Y,\U_Y)$ be a uniform space and $V \in \U_H$.
\begin{enumerate}
\item[(a)]
Two functions $f,g\: X \to Y$ \,are said to be \textit{$V$-close on a set}
$A \sbs X$ if $\bigl(f(a),g(a)\bigr) \in V$ for all $a \in A$.
\item[(b)]
A~$k$-tuple $\fb = (f_1,\dots,f_k)$ of functions $f_i\: X \to Y$ is a
\textit{$V$-solution} of the GFE \eqref{GFE0} \textit{on a set} $S \sbs X^p$ if
$$
\bigl(\wtl F(\fb \ox \alfab)(\xb),\,\wtl G(\fb \ox \betab)(\xb)\bigr) \in V
$$
for all $\xb \in S$; $\fb$ is a \textit{$V$-solution} the GFE \eqref{GFE0}
\textit{on a set} $A \sbs X$ if it is its $V$-solution on $A^p$; $\fb$ is a
\textit{$V$-solution} of the system of GFEs~\eqref{GFElam} on $A \sbs X$ if
it is a $V$-solution of every equation in the system on $A$.
\end{enumerate}
\end{defn}

For brevity's sake we say that a function $\fb = (f_1,\dots,f_k)\: X \to Y^k$
has any of the just introduced $\Gam$-continuity properties if and only if
each particular function $f_i$ has the corresponding property. Similarly,
$\fb$ is $R$-bounded (on a set $A \sbs X$) if and only each function $f_i$ is.
We say that two such functions $\fb,\gb\:X \to Y^k$ are $V$-close on $A \sbs X$
if $f_i$, $g_i$ are $V$-close on $A$ for each $i \le k$.

The system of all nonempty compact sets of a topological space $(X,\T_X)$ is
denoted by $\K(X)$.

\begin{thm}\label{thm4.4}
Let $(X,\T_X)$ be a locally compact topological space, $(Y,\U_Y)$ be a uniform
space, $\Gam\: X \cx \B \to \T_X$ be a $(\T_X,\U_Y)$ continuity scale and
\hbox{$R \sbs X \cx Y$} be a stalkwise compact bounding relation. Assume that
all the functional coefficients $F_\lam\: Y^{k \cx m_\lam} \cx X^{p_\lam} \to Y$,
$G_\lam\: Y^{k \cx n_\lam} \cx X^{p_\lam} \to Y$ in the system of
GFEs~\eqref{GFElam} are continuous in the matrix variables $y_{ij}$. Then
for each pair $D \in \K(X)$, $V \in \U_Y$ there exists a pair $C \in \K(X)$,
$U \in \U_Y$ such that $D \sbs C$ and the following implication holds true:

If a $U$-solution $\fb = (f_1,\dots,f_k)\: X \to Y^k$ of the
system~\eqref{GFElam} on $C$ is both $(\Gam,U)$-precontinuous and
$R$-bounded on $C$, then there exists a continuous solution
$\phib = (\phi_1,\dots,\phi_k)$ of the system, such that $\fb$, $\phib$
are $V$-close on $D$.
\end{thm}

\begin{proof}
Let $(X,\T_X)$, $(Y,\U_Y)$, $\Gam\: X \cx \B \to \T_X$, $R \sbs X \cx Y$,
as well as the system of GFEs \eqref{GFElam} satisfy the assumptions of the
theorem. Then $(X,\T_X)$ is completely regular, as well, hence its topology is
induced by some uniformity $\U_X$. Admit, in order to obtain a contradiction,
that there is a pair $D \in \K(X)$, $V \in \U_Y$ for which the conclusion of
the theorem fails. For each pair $C \in \K(X)$, $U \in \B$ such that $C \sps D$
we denote by $\F(C,U)$ the set of all $U$-solutions
$\fb = (f_1,\dots,f_k)\: X \to Y^k$ of the system of GFEs~\eqref{GFElam}
on $C$ which are both $(\Gam,U)$-precontinuous and $R$-bounded on $C$,
nonetheless, $\fb$ is not $V$-close on $D$ to any continuous solution
$\phib = (\phi_1,\dots,\phi_k)$ of the system~\eqref{GFElam}. According to
our assumption, all the sets $\F(C,U)$ are nonempty, and, for all
$C,C' \in \K(X)$, $U,U' \in \B$, we obviously have
$$
D \sbs C \sbs C'\ \&\ U' \sbs U \Imp \F(C',U') \sbs \F(C,U)\,.
$$

Let us embed the situation into a sufficiently saturated nonstandard
universe. More precisely, we assume that it is $\kappa$-saturated for
some uncountable cardinal $\kappa$ such that $\card \B < \kappa$, and
that there is a cofinal subset $\C \sbs \K(X)$ such that $D \sbs C$ for
each $C \in \C$ and $\card\C < \kappa$. Then
$$
\bigcap_{C \in \C,\,U \in \B} {}^*\F(C,U) \ne \emptyset\,.
$$
Let $\fb = (f_1,\dots,f_k)$ belong to this intersection. Then
$\fb\: \Xst \to \Yst^k$ is an internal function, for all $U \in \U_Y$,
$C \in \C$, $\fb$ is ${}^{*\!}(\Gam,U)$-precontinuous and $\Rst$-bounded
on $\Cst$ and it satisfies
$$
\bigl(\Fst_\lam\bigl((\fb,\gbs_\lam) \ox \alfabs_\lam\bigr)(\xb),\,
\Gst_\lam\bigl((\fb,\gbs_\lam) \ox \betabs_\lam\bigr)(\xb)\bigr) \in \Ust
$$
for any $\lam \in \Lam$ and $\xb \in \Cst^{p_\lam}$. Since $X$ is locally
compact, $\Ns(\Xst) = \bigcup_{C \in \C} \Cst$. It follows that $\fb$
is $N\!S$-continuous and almost satisfies the system~\eqref{GFElam} on
$\Ns(\Xst)$. Finally, $\fb(x) \in \bigl(\Rst[x]\bigr)^k$ for any $C \in \C$
and $x \in \Cst$. As $R[x]$ is compact for $x \in X$, in that case we have
$\fb(x) \in \bigl(\Rst[x]\bigr)^k \sbs \Ns\bigl(\Yst^k\bigr)$. Thus $\fb$
is nearstandard. According to Corollary~\ref{cor3.2}, there is a continuous
solution $\phib = (\phi_1,\dots,\phi_k)$ of the system~\eqref{GFElam}, such
that $\fb(x) \apr \phib(x)$ each $x \in X$. On the other hand, $\phibs$ is
$Q$-continuous (i.e., ${}^*$continuous), hence $\fb$ and $\phibs$ cannot be
$\Vst$-close on $\Dst$. Thus there are an $i \le k$ and an $x \in \Dst$ such
that $\bigl(f_i(x),\phist_i(x)\bigr) \nin \Vst$. However, as $D$ is compact,
$\Dst \sbs \Ns(\Xst)$. Since both $f_i$ and $\phist_i$ are $N\!S$-continuous,
taking an $x_0 \in X$ such that $x \apr x_0$, we obtain
$$
\phist_i(x) \apr \phi_i(x_0) \apr f_i(x_0) \apr f_i(x)\,,
$$
i.e., a contradiction.
\end{proof}

Like in Theorem~\ref{thm4.4}, we assume in the next three Corollaries that all
the mappings $F_\lam$, $G_\lam$ in the system of GFEs~\eqref{GFElam} are
continuous in the matrix variables $y_{ij}$ (but, for brevity's sake, we do not
mention that explicitly). In the fourth Corollary~\ref{cor4.8} this assumption
is superfluous as it is satisfied automatically.

If $(Y,\U_Y)$ is compact, then $R = X \cx Y$ is a stalkwise compact bounding
relation such that every function $\fb\: X \to Y^k$ is $R$-bounded. This makes
possible to avoid mentioning any bounding relation in the formulation of
Theorem~\ref{thm4.4}.

\begin{cor}\label{cor4.5}
Let $(X,\T_X)$ be a locally compact topological space, $(Y,\U_Y)$ be a compact
uniform space and $\Gam\: X \cx \B \to \T_X$ be a $(\T_X,\U_Y)$ continuity
scale. Then for each pair $D \in \K(X)$, $V \in \U_Y$ there is a pair
$C \in \K(X)$, $U \in \U_Y$ such that $D \sbs C$ and the following implication
holds true:

If a $U$-solution $\fb = (f_1,\dots,f_k)\: X \to Y^k$ of the system of
GFEs~\eqref{GFElam} on $C$ is $(\Gam,U)$-precontinuous on $C$, then there
exists a continuous solution $\phib = (\phi_1,\dots,\phi_k)$ of the system,
such that $\fb$, $\phib$ are $V$-close on $D$.
\end{cor}

If $(X,\T_X)$ is compact, then its topology is induced by a unique uniformity
$\U_X$ and, at the same time, it is enough to control the continuity of
functions $\fb\: X \to Y^k$ by means of a uniform continuity scale. Choosing
$D = X$ we get the following {\it global\/} version of Theorem~\ref{thm4.4}.

\begin{cor}\label{cor4.6}
Let $(X,\U_X)$ be a compact and $(Y,\U_Y)$ be an arbitrary uniform space,
$\Gam\: \B \to \U_X$ be a $(\U_X,\U_Y)$ uniform continuity scale and
\hbox{$R \sbs X \cx Y$} be a stalkwise compact bounding relation. Then for each
$V \in \U_Y$ there is a $U \in \U_Y$ such that the following implication holds
true:

If $\fb = (f_1,\dots,f_k)\: X \to Y^k$ is a uniformly $(\Gam,U)$-precontinuous
and $R$-bounded $U$-solution of the system of GFEs~\eqref{GFElam}, then there
exists a (uniformly) continuous solution $\phib = (\phi_1,\dots,\phi_k)$ of the
system, such that $\fb$, $\phib$ are $V$-close on $X$.
\end{cor}

Under the assumptions of both \ref{cor4.5} and \ref{cor4.6} we have

\begin{cor}\label{cor4.7}
Let $(X,\U_X)$, $(Y,\U_Y)$ be compact uniform spaces and $\Gam\: \B \to \U_X$
be a $(\U_X,\U_Y)$ uniform continuity scale. Then for each $V \in \U_Y$ there
is a $U \in \U_Y$ such that the following implication holds true:

If $\fb = (f_1,\dots,f_k)\: X \to Y^k$ is a uniformly $(\Gam,U)$-precontinuous
$U$-solution of the system of GFEs~\eqref{GFElam}, then there exists a
(uniformly) continuous solution $\phib = (\phi_1,\dots,\phi_k)$ of the system,
such that $\fb$, $\phib$ are $V$-close on $X$.
\end{cor}

The interested reader can easily formulate the metric versions of
Theorem~\ref{thm4.4}, as well as of Corollaries~\ref{cor4.5}\,--\,\ref{cor4.7}.

Endowing both the sets $X$, $Y$ with discrete topologies (uniformities), all
the functions $X \to Y$ become (uniformly) continuous. Then compact subsets
of $X$ are just the finite ones and, similarly, a stalkwise compact bounding
relation $R \sbs X \cx Y$ is a stalkwise finite one. In that case, choosing
$U = \Id_Y$ in Theorem~\ref{thm4.4}, we obtain the following result on
extendability of functions satisfying a system of GFEs~\eqref{GFElam} on
some finite set to its (global) solutions.

\begin{cor}\label{cor4.8}
Let $X$ and $Y$ be arbitrary sets and $R \sbs X \cx Y$ be a stalkwise finite
bounding relation. Then for each finite set $D \sbs X$ there is a finite set
$C \sbs X$ such that $D \sbs C$ and for every $R$-bounded partial solution
$\fb = (f_1,\dots,f_k)\: X \to Y^k$ of the system of GFEs~\eqref{GFElam} on
$C$ there exists a solution $\phib = (\phi_1,\dots,\phi_k)$ of the system,
such that $\phib(x) = \fb(x)$ for all $x \in D$.
\end{cor}

If the arity numbers $p_\lam$ in the system of GFEs \eqref{GFElam} have
a common upper bound~$p$, then all the particular equations in the system
can be considered as being of types $(k,m_\lam,n_\lam,p)$. In such a case,
given a $U \in \U_Y$, we say that a function $\fb\: X \to Y^k$ is a
\textit{$U$-solution} of the system \eqref{GFElam} \textit{on a set}
$S \sbs X^p$ if it is a $U$-solution of each its particular equation on $S$.
Then we have the following variant of Theorem~\ref{thm4.4}. Its proof can be
obtained by slight modifications of the proof of Theorem~\ref{thm4.4} and is
left to the reader.

\begin{thm}\label{thm4.9}
Let $(X,T_X)$ be any topological space,  $(Y,\U_Y)$ be uniform space,
$\Gam\: X \cx \B \to \T_X$ be a $(\T_X,\U_Y)$ continuity scale
and \hbox{$R \sbs X \cx Y^k$} be a stalkwise compact bounding relation.
Assume that all the equations in the system of GFEs~\eqref{GFElam} have
the same arity $p_\lam = p$, $S$ is a locally compact subspace of \,$X^p$
and each of the maps $F_\lam\: Y^{k \cx m_\lam} \cx X^p \to Y$,
$G_\lam\: Y^{k \cx n_\lam} \cx X^p \to Y$ is continuous in the matrix
variables $y_{ij}$. Then for each pair $D \in \K(X)$, $V \in \U_Y$, such that
$D^p \sbs S$, there is a triple $C \in \K(X)$, $K \in \K\bigl(X^p\bigr)$,
$U \in \U_Y$, such that $D \sbs C$, $D^p \sbs K \sbs S$ and the following
implication holds true:

If a $U$-solution $\fb = (f_1,\dots,f_k)\: X \to Y^k$ of the
system~\eqref{GFElam} on $K$ is both $(\Gam,U)$-precontinuous and
$R$-bounded on $C$, then there exists a continuous solution
$\phib = (\phi_1,\dots,\phi_k)$ of the system on $S$, such that
$\fb$, $\phib$ are $V$-close on $D$.
\end{thm}

Formulation of the corresponding modified versions of Corollaries
\ref{cor4.5}\,--\,\ref{cor4.8} is left to the reader, as well.

Comparing the ``local'' stability Theorems~\ref{thm4.4}, \ref{thm4.9} and
Corollaries~\ref{cor4.5}, \ref{cor4.5}, \ref{cor4.8} with ``global''
Corollaries~\ref{cor4.6}, \ref{cor4.7} and other global stability results
we see that while global stability deals with approximation of functions
\hbox{$\fb = (f_1,\dots,f_k)\: X \to Y^k$} by continuous solutions
$\phib = (\phi_1,\dots\phi_k)$ of the given (system of) functional equation(s)
on the whole space $X$, local stability deals with approximate extension
(and if $Y$ is discrete then right by extension) of restrictions
$\fb\rst D = (f_1\rst D,\dots,f_k\rst D)$ of such functions to some
(in the present setting compact) subset $D \sbs X$ to continuous solutions
of the (system of) functional equation(s).

The interested reader can find a brief discussion of the role of nonstandard
analysis in establishing our results as well of the possibility to replace it
by some standard methods in the final part of \cite{SlZ}.

\begin{finrem}
The general form of functional equations introduced in Section~1 was designed
with the aim to prove the stability Theorems \ref{thm4.4}, \ref{thm4.9} for
all of them in a uniform way. I expected that in order to achieve this goal
it will be necessary to assume that all the functional coefficients $F$, $G$,
$\alfa_i$, $\beta_j$ are continuous (in all their variables). Having succeeded
just with the continuity of $F$ and $G$ in the ``matrix'' variables
$y_{ij} \in Y$, without requiring their continuity in the remaining variables
$x_i \in X$, and, at the same time, without any continuity assumption on the
operations $\alfa_i$, $\beta_j$, was a big surprise for me.

A revision of the results established in \cite{SlZ}, \cite{SpZ}, \cite{Z1},
\cite{Z2}, \cite{Z3} from such a point of view reveals that in most of them
some continuity assumptions can be omitted. For instance, Theorem~3 from
\cite{Z2} (as well as Theorem~2.6 from \cite {SlZ}) on stability of continuous
homomorphisms from a locally compact topological group $G$ into any topological
group $H$ remains true \textit{without assuming that $G$ is a topological
group}. It suffices that $G$ be both a group and a locally compact topological
space. Similarly, Theorem~3.1 from \cite{Z3} on stability of continuous
homomorphisms from a locally compact topological algebra $\AAf$ into a completely
regular topological algebra $\BBf$ remains true for any universal algebra $\AAf$
endowed with a locally compact (Hausdorff) topology, without assuming
continuity of the operations in $\AAf$.

Theorems \ref{thm4.4}, \ref{thm4.9} also show that both the above mentioned
results admit a generalization in yet another direction, for the former one
stated already in Theorem~2.6 in \cite{SlZ}. Namely for a mapping $f\: G \to H$
or $f\: A \to B$ in order to be close to a continuous homomorphism it is
\textit{not} necessary to assume that it is $\Gam$-continuous with respect to
the given continuity scale $\Gam$ (as both the above mentioned theorems in
\cite{Z2} and \cite{Z3} do); it is enough that $f$ be $(\Gam,U)$ precontinuous
for a sufficiently small entourage~$U$.

On the other hand, as shown by several counterexamples in \cite{SpZ} and
\cite{Z2}, even in those weaker results one cannot manage without the control
of the examined functions by means of some continuity scale and a stalkwise
compact bounding relation. The more interesting are then the stability results
not requiring the continuity scale and/or the bounding bounding relation in
their formulation. This is, e.g., the case of the global stability result for
homomorphisms from amenable groups into the group of unitary operators on a
Hilbert space in \cite{K} (covering many more specific results proved both
before and afterwards), as well as of the local stability result for
homomorphisms from amenable groups into the unit circle $\bT$ in \cite{Z2}.
\end{finrem}


\newpage

\begin{thebibliography}{AFHL}

\bibitem[AFHL]{AFHL}
S.~Albeverio, J.\,E.~Fenstad, R.~H{\o}egh-Krohn and T.~Lindstr{\o}m,
\textit{Nonstandard methods in stochastic analysis and mathematical physics},
Academic Press, London-New York-Orlando, 1986.

\bibitem[An]{An}
R.\,M.~Anderson,
\textit{``Almost'' implies ``near''},
Trans. Amer. Math. Soc. \textbf{196} (1986), 229--237.

\bibitem[ACH]{ACH}
L.\,O.~Arkeryd, N.\,J.~Cutland and C.\,W.~Henson (eds.),
\textit{Nonstandard analysis, theory and applications},
Kluwer Academic Publishers, Dordrecht-Boston-London, 1997.

\bibitem[BB]{BB}
H.~Boualem, R.~Brouzet,
\textit{On what is the Almost-Near Principle},
Amer. Math. Monthly \textbf{119(5)} (2012), 381--393.

\bibitem[CK]{CK}
C.\,C.~Chang, H.\,J.~Keisler,
\textit{Model Theory},
North-Holland, Amsterdam-London-New York, 1973.

\bibitem[Cz]{Cz}
S.~Czerwik,
\textit{Functional Equations and Inequalities in Several Variables},
World Scientific, Singapore, 2002.

\bibitem[D]{D}
M.~Davis,
\textit{Applied Nonstandard Analysis},
John Wiley \& Sons, New York-London, 1977.

\bibitem[E]{E}
R.~Engelking,
\textit{General Topology},
PWN -- Polish Scientific Publishers, Warszawa, 1977.

\bibitem[Fo]{Fo}
G.\,L.~Forti,
\textit{Hyers-Ulam stability of functional equations
       in several variables}
Aequationes Math. \textbf{50} (1995), 143--190.

\bibitem[G]{G}
G.~Gr\"atzer,
\textit{Universal Algebra},
Van~Nostrand, Princeton-Toronto-London-Melbourne, 1968.

\bibitem[H]{H}
C.\,W.~Henson,
\textit{Foundations of nonstandard analysis:
        A~gentle introduction to nonstandard extensions},
in \cite{ACH}, 1--49.

\bibitem[HS]{HS}
P.\,J.~Hilton, U.~Stammbach
\textit{A Course in Homological Algebra},
Springer Verlag, New York-Heidelberg-Berlin, 1970.

\bibitem[HIR]{HIR}
D.\,H.~Hyers, G.~Isac, T.\,M.~Rassias,
\textit{Stability of functional equations in several variables},
Birkh\"auser Verlag, Basel-Boston, 1998.

\bibitem[HR]{HR}
D.\,H.~Hyers, T.\,M.~Rassias,
\textit{Approximate homomorphisms},
Aequationes Math. \textbf{44} (1992), 125--153.

\bibitem[K]{K}
D.~Kazhdan,
\textit{On $\eps$-representations},
Israel J. Math. \textbf{43} (1982), 315--323.

\bibitem[L]{L}
P.\,A.~Loeb,
\textit{Nonstandard analysis and topology},
in \cite{ACH}, 77--89.

\bibitem[MZ]{MZ}
M.~Ma\v{c}aj, P.~Zlato\v{s},
\textit{Approximate extension of partial $\eps$-characters of abelian groups
        to characters with application to integral point lattices},
Indag. Math. \textbf{16} (2005) 237--250.

\bibitem[Ml]{Ml}
R.\,D.~ Mauldin,
\textit{The Scottish Book},
Birkh\"{a}user Verlag, Boston, 1981.

\bibitem[Ra1]{Ra1}
T.\,M.~Rassias,
\textit{On the stability of functional equations and a problem of Ulam},
Acta Applicanda Math. \textbf{62} (2000), 23--130.

\bibitem[Ra2]{Ra2}
T.\,M.~Rassias (ed.),
\textit{Functional Equations and Inequalities},
Math. Appl. \textbf{518} (2000),
Kluwer Academic Publishers, Dordrecht, 2000.

\bibitem[R]{R}
A.~Robinson,
\textit{Non-standard Analysis} (revised ed.),
Princeton University Press Princeton, N.\,J., 1996.

\bibitem[SlZ]{SlZ}
F.~Sl\'{a}dek, P.~Zlato\v{s},
\textit{A local stability principle for continuous group homomorphisms
        in nonstandard setting},
Aequationes Math. \textbf{89} (2015), 991--1001.

\bibitem[SpZ]{SpZ}
J.~\v{S}pakula, P.~Zlato\v{s},
\textit{Almost homomorphisms of compact groups},
Illinois J.~Math. \textbf{48} (2004), 1183--1189.

\bibitem[Sk]{Sk}
L.~Sz\'{e}kelyhidi,
\textit{Ulam's problem, Hyers's solution\,---\,and where they led},
in~\cite{Ra2}, 259--285.

\bibitem[U1]{U1}
S.\,M.~Ulam,
\textit{A Collection of Mathematical Problems},
Interscience Publications, New York, 1961.

\bibitem[U2]{U2}
S.\,M.~Ulam,
\textit{Problems in Modern Mathematics},
Wiley, New York, 1964.

\bibitem[Z1]{Z1}
P.~Zlato\v{s},
\textit{Stability of homomorphisms between compact algebras},
Acta Univ. Mathaei Beli, ser. Math. \textbf{15} (2009), 73--78.

\bibitem[Z2]{Z2}
P.~Zlato\v{s},
\textit{Stability of group homomorphisms in the compact-open topology},
J.~Logic \& Analysis \textbf{2:3} (2010), 1--15.

\bibitem[Z3]{Z3}
P.~Zlato\v{s},
\textit{Stability of homomorphisms in the compact-open topology},
Algebra Universalis \textbf{64} (2010), 203--212.

\end{thebibliography}
\end{document}